\documentclass[reqno,12pt,epsf]{amsart}
\usepackage{amsmath}
\usepackage{latexsym,amsmath}
\usepackage{amsthm}
\usepackage{amssymb}
\usepackage{graphics}

\newtheorem{myproposition}{Proposition}[section]
\newtheorem*{mytheorem}{Theorem 1}

\newtheorem{mycorollary}[myproposition]{Corollary}

\newtheorem*{myremark}{Remark}

\newcounter{rot}

\title[Total Vertex Irregularity Strength of Forests]
{Total Vertex Irregularity Strength of Forests}

\author[M. Anholcer]{Marcin Anholcer}
\curraddr[M. Anholcer]{Pozna\'n University of Economics\\ Faculty of Informatics and Electronic Economy\\ Pozna\'n, Poland}
\email{m.anholcer@ue.poznan.pl}
\author[M. Karo\'nski]{Micha{\l} Karo\'nski}
\curraddr[M. Karo\'nski]{Adam Mickiewicz University\\ Faculty of Mathematics and
 Computer Science\\ Pozna\'n, Poland\\ and Emory University\\
Department of Mathematics and Computer Science\\
Atlanta, GA, USA}
\email{karonski@amu.edu.pl}
\author[F. Pfender]{Florian Pfender}
\curraddr[F. Pfender]{Universit\"at Rostock\\ Institut f\"ur Mathematik\\ Rostock, Germany}
\email{Florian.Pfender@uni-rostock.de}

\begin{document}

\begin{abstract}
We investigate a graph parameter called the \textit{total vertex irregularity strength} ($tvs(G)$), i.e. the minimal $s$  such that there is a labeling $w: E(G)\cup V(G)\rightarrow \{1,2,\dots,s\}$ of the edges and vertices of $G$ giving distinct weighted degrees $wt_G(v):=w(v)+\sum_{v\in e \in E(G)}w(e)$ for every pair of vertices of $G$. We prove that $tvs(F)=\lceil (n_1+1)/2 \rceil$ for every forest $F$ with no  vertices of degree $2$ and no isolated vertices, where $n_1$ is the number of pendant vertices in $F$. Stronger results for trees were recently proved by Nurdin et al.
\end{abstract}

\maketitle
\section{Introduction}
Let us consider the simple undirected graph $G=(V(G),E(G))$ without loops, without isolated edges and with at most one isolated vertex. We assign a label (a natural positive number) to every edge (denoted by $w(e)$ for all $e \in E(G)$) and to every vertex (denoted by $w(v)$ for all $v \in V(G)$). We will refer to such a labeling as a \textit{total weighting} of $G$. For every vertex $v \in V(G)$, we define its \textit{weighted degree} as
\[
wt_G(v)=\sum_{e\ni v}w(e)+w(v).
\]

We call the labeling $w$ \textit{irregular} if for each pair of vertices, their weighted degrees are distinct. In~\cite{ref_BacJenMilRya1}, a new graph parameter called \textit{total vertex irregularity strength} ($tvs(G)$) was defined as   the smallest integer $s$ such that there exists a total weighting of $G$ with integers $\{1,2,\dots,s\}$ that is irregular. This parameter is similar to the irregularity strength of $G$ ($s(G)$), introduced in \cite{ref_ChaJacLehOelRuiSab1} (see also  \cite{ref_Tog1}, \cite{ref_FauLeh1}, \cite{ref_FriGouKar1} and \cite{ref_Nie1}), where only weights on the edges are allowed.

In~\cite{ref_BacJenMilRya1}, several bounds and exact values of $tvs(G)$ were established for different types of graphs. In particular, the authors proved that for every graph $G$ with $n$ vertices and $m$ edges, the following bounds  hold.
\begin{equation}\label{Jendrol_bound1}
\left\lceil \tfrac{n+\delta(G)}{\Delta(G)+1} \right\rceil \leq tvs(G) \leq n+\Delta(G)-2\delta(G)+1,
\end{equation}
where $\Delta (G)$ and $\delta (G)$ are the maximum and the minimum degree of $G$, respectively.

Recently, a much stronger upper bound on  $tvs(G)$ has been established in \cite{ref_AnhKalPrz1}. Namely, for every graph $G$ with $\delta(G)>0$,
\[
tvs(G)\leq \left\lceil \tfrac{3n}{\delta} \right\rceil+1.
\]

 One should also mention that in \cite{ref_BacJenMilRya1},  exact values of $tvs(G)$ for stars, cliques and prisms are given. Furthermore it is shown that for every tree $T$ without vertices of degree two, the following bounds hold.
\begin{equation}\label{Jendrol_bound4}
\left\lceil{\tfrac{n_1+1}{2}}\right\rceil\leq tvs(T) \leq n_1,
\end{equation}
where $n_1$ is the number of pendant vertices of $T$.

Our main result stated below shows that the trivial  lower bound in (\ref{Jendrol_bound4}) is the true value of $tvs(T)$. Note the easy observation that for trees with less than $4$ vertices, the equality $tvs(T)=\left\lceil{\frac{n_1+1}{2}}\right\rceil$ holds.

\begin{mytheorem}\label{main_thm}
For every forest $F$ with $n_1$ vertices of degree one and with no vertices of degree two and no isolated vertices,
\[
tvs(F)=\left\lceil {\tfrac{n_1+1}{2}} \right\rceil.
\]
\end{mytheorem}

The proof of Theorem~\ref{main_thm} is given in the next section.

\begin{myremark}
Recently, stronger results for trees were proved by Nurdin et al. (\cite{ref_Nur}). However we decided to publish our paper for two reasons. Firstly, we consider more general case of forests, not only trees. Secondly, we use different proof technique. 
\end{myremark}

\section{Proof of  Theorem~\ref{main_thm}}\label{lab_sec_thm}

Let us consider a forest $F$. Denote by  $V_i$ the set of vertices of degree $i$, and let $n_i=|V_i|$. So $V_1$ is the set of pendant vertices and we call edges incident to  pendant vertices  {\it pendant edges}.
Denote by $C_{ij}$ the set of  its vertices of degree
$i$ with exactly $j$ pendant neighbors, where $i,j \geq 0$, and let $n_{ij}=|C_{ij}|$.

 Assume that in a total weighting of  $F$ we use labels (weights) from the set $\{1,2,\dots,s\}$. Then, the lowest and the highest weighted degree of a pendant vertex $v\in V_1$ can be $2$ and $2s$, respectively. Since there are $n_1$ such vertices, and each vertex has to have different total weight, the lower bound of (\ref{Jendrol_bound4}) trivially follows and extends to all forests, i.e.,
\begin{equation}\label{forest_lower_bound}
tvs(F) \geq \left\lceil{\tfrac{n_1+1}{2}}\right\rceil.
\end{equation}
To prove Theorem \ref{main_thm}, it is sufficient to construct an irregular labeling of $F$ using  elements from the set $\{1,2,\dots,s\}$ only, where $s:=\left\lceil{\tfrac{n_1+1}{2}}\right\rceil$.

We will often use the well-known fact that in every tree $T$ with maximum degree $\Delta\ge 1$ the numbers $n_j(T)$ of vertices of degree $j$ satisfy the equation
\begin{equation}
n_1(T)=2+\sum_{j=3}^{\Delta}{(j-2)n_j(T)}.\label{fact_degrees}
\end{equation}
Further, for forests $F$ with $n\ge 3$ and $n_2\le 1$, we have
\[
2n_{30}+n_{43}+2n_{44}\le n_1-2,
\]
with equality only for $K_{1,4}$ and $P_3$. This can be seen as follows. For the sake of contradiction, assume that $F$ is a minimal counter example to the inequality. Then $n_{43}=0$ as otherwise we could delete a pendant vertex adjacent to a vertex of class $C_{43}$ and receive a smaller counter example. Further, $n_{44}=0$ as otherwise we can delete a vertex of $C_{44}$ and its neighbors and either receive a smaller counterexample, or $F$ itself is a $K_{1,4}$, which is not a counterexample.
Now delete all pendant vertices from $F$ to construct a forest $F'$ with $n_3'\ge n_{30}$ and $n'=n-n_1$ vertices. Then
\[
2n_{30}\le 2n_3'\le_{\eqref{fact_degrees}} n'-2=n-n_1-2\le_{\eqref{fact_degrees}} n_1+n_2-4\le n_1-3.
\]
As Theorem~\ref{main_thm} is easily verified for $K_{1,4}$ and $P_3$, we may later work with the inequality
\begin{equation}\label{fact2_degrees}
2n_{30}+n_{43}+2n_{44}\le n_1-3.
\end{equation}

Label all non-pendant edges with $s$. Next, we will label the pendant edges.
First, label half the isolated edges (rounded down) with $1$ and the remaining isolated edges with $s$.
Let now $v\in C_{kj}$ for some $k\ge 3$ and $j\ge 1$, and we will label the incident pendant edges. Order the values $\{ 1,\ldots,s-1\}$ as a list $S=(1,s-1,2,s-2,3,s-3,\ldots)$.
\begin{enumerate}
 \item If $j$ is even, label $j/2$ pendant edges with $s$.
 \item If $j$ is odd, label $(j-1)/2$ (variant 1) or $(j+1)/2$ (variant 2) pendant edges with $s$.
 \item The remaining pendant $2i+\delta$ ($0\le \delta\le 1$) edges are labeled from $S$, where we use the first $2i$ and the last $\delta$ values in $S$, which have not previously been used on non-isolated edges. 
\end{enumerate}
During the process, choose variant $1$ and variant $2$, so that the number of pendant edges labeled $s$ is maximized but at most $s$, so there are $s-2$ or $s-1$ pendant edges labeled with a number less than $s$.

For notation, we write $C_{kj}^i$ for the vertices of variant $i$ in $C_{kj}$, and $n_{kj}^i$ for their number.

Note that in this labeling, regardless of the labels in $\{ 1,\ldots,s\}$ we give to the vertices themselves, 
vertices in $C_1$ have total weights between $2$ and $2s$, and all other vertices have total weights of at least $2s+2$.
%
Label every pendant vertex incident to a non-isolated edge labeled with a number less than $s$ with $1$. This guarantees that the total weights of all these vertices are different. The remaining vertices in $C_1$ can now be weighted greedily one-by-one.

For the weight range $2s+2\le wt(v)\le 3s$, only vertices in $C_{31}^1\cup C_{32}\cup C_{33}$ play a role. Vertices $v\in C_{33}^2$ have total weight $2s+w(v)$, all others have weight $2s+w(e)+w(v)$, where $e$ is a pendant edge with label other than $s$ incident to $v$.
We can set $w(v)=1$ for at most $s-1-n_{33}^2$ vertices $v\in (C_{31}^1\cup C_{32}\cup C_{33}^1)$, giving them all different total weights in this weight range, and greedily choose weights for the vertices in $C_{33}^2$ to fill out the remaining weights. Note that $n_{33}^2\le \frac{n_{33}+1}{2}\le \frac{n_1+3}{6}\le \frac{s+1}{3}$,  so this is possible.

For the weight range $3s+1\le wt(v)\le 4s$, only vertices in
\[
 C_{3}\cup C_{41}^1\cup C_{42}\cup C_{43}\cup C_{44}\cup C_{55}^1
\]
 play a role.
Label all vertices $v\in C_{41}^1\cup C_{42}\cup C_{43}^2\cup C_{55}^1$ with $w(v)=s$,
giving them total weight $wt(v)=4s+w(e)$, and thus pairwise different weights outside the weight range we currently consider.
We give the remaining (at most) $\max\{ 1,n_{31}^1+n_{32}+n_{33}-s+1\}$ vertices from $C_{31}^1\cup C_{32}\cup C_{33}^1$ weight $w(v)=s$, giving them all different total weights $3s+w(e)$ in this range. Vertices $v\in C_{30}\cup C_{43}^1\cup C_{44}$ have total weight $3s+w(v)$, so one can greedily fit them into the remaining weights of the range, provided there is enough room. If
$n_{31}+n_{32}+n_{33}\ge s$, then
\begin{multline*}
 n_{30}+n_{43}^1+n_{44}+\max\{ 1,n_{31}^1+n_{32}+n_{33}-s+1\}\\
 \le n_3+n_4-s+1\le_{\eqref{fact_degrees}} n_1-s-1\le s-1.
\end{multline*}
If, on the other hand, $n_{31}+n_{32}+n_{33}\le s-1$, then
\begin{multline*}
n_{30}+n_{43}^1+n_{44}+\max\{ 1,n_{31}^1+n_{32}+n_{33}-s+1\}\\
= n_{30}+n_{43}^1+n_{44}+1\le_{\eqref{fact2_degrees}} \left\lfloor\tfrac{n_1}{2}\right\rfloor=s-1.
\end{multline*}
Thus, there is enough room to fit all the mentioned vertices.

For the weight range $wt(v)\ge 4s+1$, only vertices in $C_j,j\ge 4$ play a role. By~\eqref{fact_degrees}, there are at most $s$ such vertices. We have dealt with vertices in $C_{4}\setminus C_{40}$ already, so all remaining vertices have $wt(v)\ge 4s+w(v)$. Thus, we have at each vertex $s$ choices for the total weight, which is enough to allow us to greedily pick the values $w(v)$ to complete the irregular total weighting.
\qed

\section{Final remarks}

Note that in the proof of Theorem~\ref{main_thm}, the total weight $2s+1$ was not used. With this observation we can prove a slight generalization.
\begin{mytheorem}
 For every forest $F$ on $n$ vertices, with $n_0=0$ and $n_2\le 1$, we have
\[
tvs(F)=\left\lceil {\tfrac{n_1+1}{2}} \right\rceil.
\]
\end{mytheorem}
\begin{proof}
 The proof is the same as above, with one extra observations. Just note in the end, that since a vertex $v$ of degree $2$ is incident to an edge of label $s$ (either a non-pendant edge or a pendant edge with this label by the construction), we can choose $w(v)$ such that $wt(v)=2s+1$.
\end{proof}
As a consequence we have the following corollary.
\begin{mycorollary}\label{prop_binary_tree}
Let $T$ be a binary tree. Then $tvs(T)=\lceil\frac{n_1+1}{2}\rceil$.
\end{mycorollary}


\begin{thebibliography}{99}
\bibitem{ref_AnhKalPrz1}
Anholcer M., Kalkowski M., Przyby{\l}o J.: "A new upper bound for the total vertex irregularity strength of graphs". Discrete Mathematics, 309 (2009) 6316-6317.

\bibitem{ref_Tog1}
Baril J.-L., Kheddouci H., Togni O., The Irregularity Strength of Circulant Graphs. {\it Discrete Mathematics}, 304 (2005), 1-10.

\bibitem{ref_BacJenMilRya1}
Baca M., Jendrol S., Miller M., Ryan J, On Irregular Total Labelings.{\it Discrete Mathematics} 307 (2007), 1378 - 1388.

\bibitem{ref_ChaJacLehOelRuiSab1}
Chartrand G., Jacobson M.S., Lehel J., Oellermann O.R., Ruiz S., Saba F.,Irregular Networks. {\it Congressus Numerantium }64 (1988), 187 - 192.

\bibitem{ref_FauLeh1}
Faudree R.J., Lehel J.,Bound on the Irregularity Strength of Regular Graphs. {\it Colloquia Mathematica Societatis J\'anos Bolyai 52, Combinatorics, Eger (Hungary)}. North - Holland, Amsterdam 1987, 239 - 246.

\bibitem{ref_FriGouKar1}
Frieze A., Goukd R.J., Karo\'nski M., Pfender F., On graph irregularity strength. {\it Journal of Graph Theory} 41, 120-137.

\bibitem{ref_Nie1}
Nierhoff T., A Tight Bound on the Irregularity Strength of Graphs.{\it SIAM Journal on Discrete Mathematics} Vol. 13 (2000), No 3, 313 - 323.

\bibitem{ref_Nur}
Nurdin, Baskoro E.T., Salman A.N.M., Gaos N.N., On the Total Vertex
Irregularity Strength of Trees, Discrete Mathematics 310 (2010), 3043-
3048.

\end{thebibliography}
\end {document}